\documentclass[12pt,reqno]{amsart}
\usepackage{amsthm}
\usepackage{fullpage,url}
\usepackage{hyperref} 
\usepackage{graphicx}
\usepackage{epstopdf}
\DeclareGraphicsExtensions{.eps}
\usepackage{amscd}   
\usepackage[all,graph,poly, knot]{xy} 
\usepackage{amssymb,amsmath}
\usepackage{pstricks,pst-plot,pst-node}
\usepackage{multicol}
\usepackage{setspace}
\DeclareFontEncoding{OT2}{}{} 

\swapnumbers
\newtheorem{theorem}{Theorem}[section]
\newtheorem{lemma}[theorem]{Lemma}

\theoremstyle{definition}
\newtheorem{definition}[theorem]{Definition}

\theoremstyle{remark}
\newtheorem{remark}[theorem]{Remark}
\newtheoremstyle{head}
{}
{}
{\bfseries}
{}
{}
{}
{.5em}
{}

\theoremstyle{head}

\newtheoremstyle{citing}
  {3pt}
  {3pt}
  {\itshape}
  {}
  {\bfseries}
  {:}
  {.5em}
  {\thmnote{#3}}

\theoremstyle{citing}
\begin{document}
\title[]{Periodic and falling-free motion\\ of an inverted spherical pendulum\\ with a moving pivot point}
\author{Ivan Polekhin}
\address{}
\email{ivanpolekhin@gmail.com}
\urladdr{}
\date{November 10, 2014}
\begin{abstract}
For the system of an inverted spherical pendulum with friction and a periodically moving pivot point we prove the existence of at least one periodic solution with the additional property of being falling-free. The last means that the pendulum never becomes horizontal along the considered periodic solution. Presented proof is an application of some recent results in the fixed point theory.
\end{abstract}
\maketitle
%
%
%
%
\section{Introduction}
\label{section-introduction}
%
%

One of the problems originally presented in \textit{What is mathematics?} book by Courant and Robbins \cite{CR} is stated as follows:
\par
\textit{Suppose a train travels from station $A$ to station $B$ along a straight 
section of track. The journey need not be of uniform speed or acceleration.
The train may act in any manner, speeding up, slowing down, 
coming to a halt, or even backing up for a while, before reaching $B$. 
But the exact motion of the train is supposed to be known in advance; 
that is, the function $s = f(t)$ is given, where $s$ is the distance of the train 
from station $A$, and $t$ is the time, measured from the instant of departure. 
On the floor of one of the cars a rod is pivoted so that it may move 
without friction either forward or backward until it touches the floor. If it 
does touch the floor, we assume that it remains on the floor henceforth; 
this will be the case if the rod does not bounce. Is it possible to place 
the rod in such a position that, if it is released at the instant when the 
train starts and allowed to move solely under the influence of gravity 
and the motion of the train, it will not fall to the floor during the entire 
journey from $A$ to $B$?}
\par
As an exercise, the authors suggest to prove a more general result where a spherical inverted pendulum is considered instead of a planar pendulum. We consider slightly different problem introducing a viscous friction force acting on the mass point of the pendulum and prove that for an arbitrarily small friction coefficient a periodic solution always exists. Moreover, we show that the considered periodic solution never approaches horizontal plane.
\par
In the first section, we present the system of governing dynamical equations for the system and prove some properties which we are going to use in the second section where we apply some recent results  by Srzednicki, W\'{o}jcik and Zgliczynski \cite{SWZ} to our system.
%
%
\section{Governing equations}
\label{first-section}
%
Let $Oxyz$ be an orthogonal moving reference frame such that $O$ coincide with the pivot point of the pendulum, $Ox$ and $Oy$ axis are in a horizontal plane and always remain parallel to themselves at the initial moment of time; $Oz$ is vertical and oriented in an opposite way to the gravitational force. By $r_{moving}$ we denote the radius vector of the mass point in $Oxyz$. Let $x$, $y$, and $z$ be its components. 
\begin{figure}[!h]
\centering
\def\svgwidth{320px}
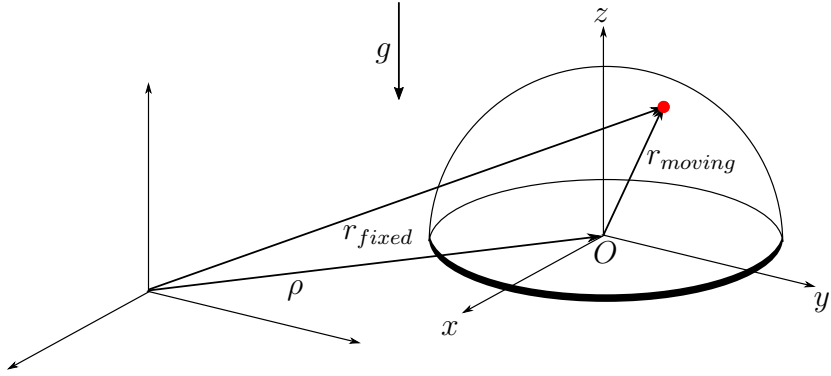
\caption{Fixed and moving reference frames}
\end{figure}
Since the pivot point moves periodically in the horizontal plane, then for the radius vector $r_{fixed}$ of the mass point in some fixed reference frame with its axis parallel to the axis of $Oxyz$, we have
\begin{equation*}
r_{fixed} = r_{moving} + \rho,
\end{equation*}
where $\rho = \xi e_x + \eta e_y$; we assume $\xi, \eta$ to be $2\pi$-periodic smooth functions.
Obviously, we have similar relations for the velocities and accelerations
\begin{equation*}
\dot r_{fixed} = \dot r_{moving} + \dot \rho, \quad \ddot r_{fixed} = \ddot r_{moving} + \ddot \rho.
\end{equation*}
The mass point moves under the action of the gravitational force, the viscous friction force, and the force of constraint
\begin{equation}
\label{eq1}
m\ddot r_{fixed} = F_{grav} + F_{friction} + N.
\end{equation}
Here $N$ is the constraint force which is parallel to the radius vector $r_{moving}$.
\begin{equation*}
N = |N|\frac{r_{moving}}{|r_{moving}|} = |N|e_n.
\end{equation*}
\begin{remark}
Forces of constraint are the forces which allows one to consider a system with constraints as constraint-free with additional unknown \textit{apriori} forces. More on constraint forces and their use in mechanics one can find in \cite{GPS}.
\end{remark}
We assume that the friction force is determined by the following model 
\begin{equation}
\label{eq2}
F_{friction} = -\gamma (\dot r_{fixed} - \dot \rho) = -\gamma \dot r_{moving}, \quad \gamma > 0.
\end{equation}
We can now rewrite (\ref{eq1}) as follows
\begin{equation}
\label{eq1mod}
m\ddot r_{moving} = |N|e_n - mg e_z  - m\ddot\rho - \gamma \dot r_{moving}.
\end{equation}
\begin{remark}
 For the above equation (\ref{eq1mod}), the solution function $r_{moving} (\cdot, r_0, \dot r_0) \colon \mathbb{R} \to S^2$ is determined by the initial conditions $r_0$ and $\dot r_0$.
\end{remark}
\par 
Since functions $\xi$ and $\eta$ are $2\pi$-periodic, then the extended phase space of our system can be considered as $\mathbb{R}/2\pi\mathbb{Z} \times T{S}^2$. Let $F \colon \mathbb{R}/2\pi\mathbb{Z} \times T{S}^2 \to \mathbb{R}$ be a function defined by the following equation
\begin{equation}
F = \frac{m}{2}(\dot r_{moving}, \dot r_{moving}).
\end{equation}
Now consider the submanifold $F = c$ of the extended phase space and show that if $c > 0$ is large, then along the solutions starting at $F = c$ the function $F$ is locally decreasing. More specifically,
\begin{lemma}
\label{lem1}
There exists $c>0$ such that
\begin{equation*}
\dot F\Big|_{F=c} < 0.
\end{equation*}
\end{lemma}
\begin{proof}
From the definition of $F$ and (\ref{eq1}) we easily have
\begin{equation*}
\begin{aligned}
\dot F  &=m(\ddot r_{moving}, \dot r_{moving}) = m(\ddot r_{fixed} - \ddot \rho, \dot r_{moving}) \\
&= ( F_{grav} + F_{friction} + N, \dot r_{moving}) - m(\ddot\rho, \dot r_{moving}).
\end{aligned}
\end{equation*}
Therefore, taking into account (\ref{eq2}) and that $|\dot r_{moving}| = (2c/m)^{1/2}$, we obtain
\begin{equation*}
\begin{aligned}
 \dot F &=  (F_{grav}, \dot r_{moving}) - \gamma |\dot r_{moving}|^2 - m(\ddot \rho, \dot r_{moving})\\
 &\leqslant|F_{grav}||\dot r_{moving}| - \gamma|\dot r_{moving}|^2+m|\ddot\rho||\dot r_{moving}| \\
 &=(2c/m)^{1/2}|F_{grav}| - 2c\gamma/m + (2cm)^{1/2}|\ddot\rho|.
\end{aligned}
\end{equation*}
Since $|\ddot\rho|$ is bounded, then for $c > 0$ sufficiently large we have $\dot F <0$. 
\end{proof}
Now we prove that if the pendulum, as well as its velocity vector, is in the horizontal plane, then at least locally the mass point is falling down, i.e. the following lemma is true.
\begin{lemma}
\label{lem2}
If the solution $r_{moving}$ of (\ref{eq1}) at the time $t$ satisfies $(r_{moving}(t), e_z) = z(t) = 0$ and $(\dot r_{moving}(t), e_z) = \dot z(t) =0$, then $(\ddot r_{moving}(t), e_z) = \ddot z(t) < 0$.
\end{lemma}
\begin{proof}
For the pendulum being in the horizontal position and moving horizontally as well, we have the following:
\begin{enumerate}
\item the friction force is always parallel to the radius vector $\dot r_{moving}$, therefore, at the given conditions, the friction force is directed horizontally;
\item the constraint force is in the horizontal plane if $z = 0$;
\item $Oxyz$ is moving along the horizontal plane, therefore, $(\dot\rho, e_z)$ = 0.
\end{enumerate}
\begin{figure}[!h]
\centering
\def\svgwidth{320px}
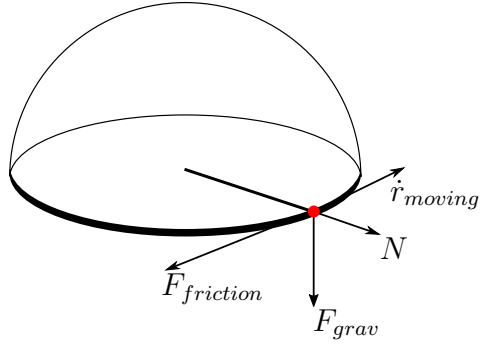
\caption{Forces acting on the mass point when the pendulum is horizontal.}
\end{figure}
Taking into account the above, from (\ref{eq1}) we obtain
\begin{equation*}
\begin{aligned}
(m\ddot r_{moving}, e_z) &= (F_{grav} + F_{friction} + N - \ddot\rho, e_z) \\
&= -mg -\gamma(\dot r_{fixed} - \dot\rho, e_z) + (N, e_z) - (\ddot\rho,e_z) = -mg < 0.
\end{aligned}
\end{equation*}
\end{proof}

%
\section{Main result}
\label{second-section}
%
In this section we are going to apply some recent developments in fixed point theory by Srzednicki, W\'{o}jcik and Zgliczynski \cite{SWZ} to our system and prove the existence of a falling-free periodic solution. First, following \cite{SWZ}, we introduce some definitions which we slightly modify for our use.
\par
From now on, we assume that $v \colon \mathbb{R}\times M \to TM$ is a smooth time-dependent vector-field on a manifold $M$. 
\begin{definition}
For $t_0 \in \mathbb{R}$ and $x_0 \in M$, the map $t \mapsto x(t,t_0,x_0)$ is the solution for the initial value problem for the system $\dot x = v(t, x)$, such that $x(0,t_0,x_0)=x_0$.
\end{definition}
\begin{definition}
Let $W \subset \mathbb{R} \times M$. Define the {exit set} $W^-$ as follows. A point $(t,x)$ is in $W^-$ if there exists $\delta>0$ such that $(t+t_0, x(t,t_0,x_0)) \notin W$ for all $t \in (0,\delta)$.
\end{definition}
\begin{definition}
We call $W \subset \mathbb{R}\times M$ a {Wa\.{z}ewski block} for the system $\dot x = v(t,x)$ if $W$ and $W^-$ are compact.
\end{definition}
 Now introduce some notations. By $\pi_1$ and $\pi_2$ we denote the projections of $\mathbb{R}\times M$ onto $\mathbb{R}$ and $M$ respectively. If $Z \subset \mathbb{R}\times M$, $t\in\mathbb{R}$, then we denote
\begin{equation*}
Z_t=\{z \in M \colon (t,z) \in Z\}.
\end{equation*}
\begin{definition}
A set $W \subset [a,b] \times M$ is called a segment over $[a,b]$ if it is a block with respect to the system $\dot x = v(t,x)$ and the following conditions hold:
\begin{itemize}
\item there exists a compact subset $W^{--}$ of $W^-$ called the essential exit set such that
\begin{equation*}
W^-=W^{--}\cup(\{b\}\times W_b),\quad W^-\cap([a,b)\times M) \subset W^{--},
\end{equation*}
\item there exists a homeomorphism $h\colon [a,b]\times W_a \to W$ such that $\pi_1 \circ h = \pi_1$ and
\begin{equation}
\label{cond-2}
h([a,b]\times W_a^{--})=W^{--}.
\end{equation}
\end{itemize}
\end{definition}
\begin{definition}
Let $W$ be a segment over $[a,b]$. It is called periodic if
\begin{equation*}
(W_a,W_a^{--})=(W_b,W_b^{--}).
\end{equation*}
\end{definition}
\begin{definition}
For periodic segment $W$, we define the corresponding monodromy map $m$ as follows
\begin{equation*}
m\colon W_a\to W_a, \quad m(x) = \pi_2 h(b,\pi_2 h^{-1}(a,x)).
\end{equation*}
\end{definition}
\begin{remark}
The monodromy map $m$ is a homeomorphism. Moreover, it can be proved that a different choice of $h$ satisfying \ref{cond-2} leads to the monodromy map homotopic to $m$. It follows that the isomorphism in homologies
\begin{equation*}
\mu_W = H(m) \colon H(W_a,W_a^{--}) \to H(W_a, W_a^{--})
\end{equation*}
is an invariant of $W$.
\end{remark}
\begin{theorem} 
\label{th1}
\cite{SWZ} Let W be a periodic segment over $[a,b]$. Then the set
\begin{equation*}
U = \{ x_0 \in W_a \colon x(t-a,a,x_0) \in W_t\setminus W_t^{--}\,\mbox{for all}\,\, t \in [a,b] \}
\end{equation*}
is open in $W_a$ and the set of fixed points of the restriction $x(b-a,a,\cdot)|_U \colon U \to W_a$ is compact. Moreover, if $W$ and $W^{--}$ are ANRs then
\begin{equation*}
\mathrm{ind}(x(b-a,a,\cdot)|_U) = \Lambda(m) - \Lambda(m|_{W_a^{--}}).
\end{equation*}
Where by $\Lambda(m)$ and $\Lambda(m|_{W_a^{--}})$ we denote the Lefschetz number of $m$ and $m|_{W_a^{--}}$ respectively. In particular, if $\Lambda(m) - \Lambda(m|_{W_a^{--}}) \ne 0$ then $x(b-a,a,\cdot)|_U$ has a fixed point in $W_a$.
\end{theorem}
Using the above theorem and the lemmas from the previous part, we can now prove that there exists a periodic falling-free solution of (\ref{eq1mod}).  
\begin{theorem}For any $\gamma > 0$ for the system (\ref{eq1mod}), there exist $r_0$ and $\dot r_0$ such that solution $r_{moving}(t, r_0, \dot r_0)$ with the initial conditions $r_{moving}(0, r_0, \dot r_0) = r_0$ and $\dot r_{moving}(0, r_0, \dot r_0) = \dot r_0$ is periodic and $(r_{moving}(t, r_0, \dot r_0), e_z) >0$ for all $t$.
\end{theorem}
\begin{proof}
Let $W \subset \mathbb{R}/2\pi\mathbb{Z} \times T{S}^2$ be a manifold with a boundary defined by the inequalities $F \leqslant c$, $z \geqslant 0$, where $c$ is obtained from lemma \ref{lem1}. One can easily prove that $W$ is diffeomorphic to $\mathbb{R}/2\pi\mathbb{Z} \times {D}^2 \times D^2$, where ${D}^2$ is a 2-dimensional disk (with boundary). Let us now prove that $W$ is a periodic segment over $[0, 2\pi]$ for (\ref{eq1}). 
\par
Indeed, from lemma \ref{lem1} we see that the essential exit set is entirely on the boundary $z = 0$. Moreover, from lemma \ref{lem2} we obtain that the points which satisfies $z = 0$ and $\dot z > 0$ do not belong to the essential exit set, at the same time, if for some point we have $z = 0$, $\dot z \leqslant 0$ then this point is in the essential exit set. Therefore, the essential exit set is  compact. Homeomorphism $h$ in our case is an identity map, i.e. $[0, 2\pi]\times W_0 = W$.
\begin{figure}[!h]
\centering
\tiny{
\def\svgwidth{320px}
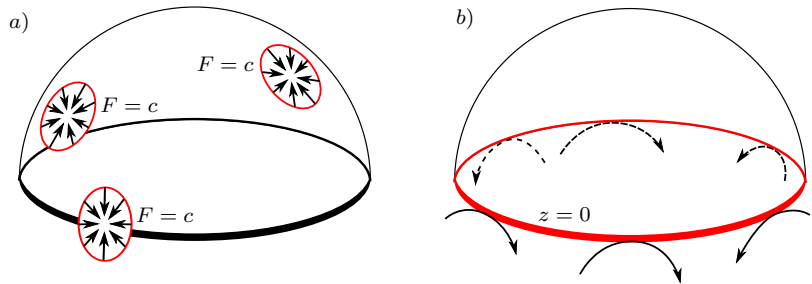
}
\caption{Main features of behaviour of the solutions at the boundary (in red) of W: a) trajectories can leave $W$ only through the $z=0$ part of the boundary b) if for a point we have $z = 0$ then the solution starting from it, leaves $W$ iff $\dot z \leqslant 0$.}
\end{figure}
\par
Finally, one can easily see that $W_0$ is homotopic to $D^2$ and $W_0^{--}$ is homotopic to ${S}^1$. Therefore, taking into account that $m = \mathrm{id}$, we get
\begin{equation*}
 \Lambda(\mathrm{id}_{W_0}) - \Lambda(\mathrm{id}_{W_0^{--}}) = \chi(D^2) - \chi({S}^1) = 1 - 0 \ne 0
 \end{equation*}
  and theorem \ref{th1} can be applied.
\end{proof}

%
%
%
%

\end{document}